\documentclass[11pt,a4paper,usenames,dvipsnames]{article}

\usepackage{url,amsmath,amssymb,latexsym,pstricks,mathrsfs,comment,amsthm,graphicx,tikz,tikz-cd,enumerate,accents,pgffor,cite,wrapfig,multicol,float,cases,calc,geometry}
\usepackage[colorlinks]{hyperref}

\usepackage[T1]{fontenc}
\usepackage{ wasysym }

\usepackage{xcolor}

\usepackage{enumitem}

\usetikzlibrary{calc}

\newcommand{\nc}{\newcommand}
\nc{\rnc}{\renewcommand}

\nc\restr{{\restriction}}
\nc\Reg{\operatorname{Reg}}
\nc\rank{\operatorname{rank}}
\nc\M{\mathcal M}
\nc\sub\subseteq

\nc\ben{\begin{enumerate}[label=\textup{(\roman*)},leftmargin=7mm]}
\nc\BEN{\begin{enumerate}[label=\textup{(\Roman*)},leftmargin=7mm]}
\nc\een{\end{enumerate}}

\nc\AND{\qquad\text{and}\qquad}
\nc{\leqR}{\leq_{\R}}
\nc{\leqL}{\leq_{\L}}
\nc{\leqJ}{\leq_{\J}}
\nc\T{\mathcal T}

\nc\Z{\mathbb Z}
\nc\N{\mathbb N}
\nc\im{\operatorname{im}}
\nc\RR{\mathbb R}

\nc{\leqRd}{\leq_{\Rd}}
\nc{\leqLd}{\leq_{\Ld}}
\nc{\leqJd}{\leq_{\Jd}}

\rnc\iff{\ \Leftrightarrow\ }
\rnc\implies{\ \Rightarrow\ }

\renewcommand{\H}{\mathrel{\mathcal H}}
\renewcommand{\L}{\mathrel{\mathcal L}}
\newcommand{\R}{\mathrel{\mathcal R}}
\newcommand{\D}{\mathrel{\mathcal D}}
\newcommand{\J}{\mathrel{\mathcal J}}
\newcommand{\G}{\mathrel{\mathcal G}}

\newcommand{\Ld}{\mathrel{\mathcal L'}}
\newcommand{\Rd}{\mathrel{\mathcal R'}}
\newcommand{\Dd}{\mathrel{\mathcal D'}}
\newcommand{\Jd}{\mathrel{\mathcal J'}}
\newcommand{\Gd}{\mathrel{\mathcal G'}}

\newtheorem{thm}{Theorem}
\newtheorem{lemma}[thm]{Lemma}
\newtheorem{cor}[thm]{Corollary}
\newtheorem{prop}[thm]{Proposition}

\theoremstyle{definition}

\newtheorem{rem}[thm]{Remark}

\newtheorem{eg}[thm]{Example}
\newtheorem{egs}[thm]{Examples}

\nc\pf{\begin{proof}}
\nc\epf{\end{proof}}
\nc\epfres{\hfill\qed}
\let\oldproofname=\proofname
\renewcommand{\proofname}{\rm\bf{\oldproofname}}

\begin{document}

\title{Green's relations and stability for subsemigroups}

\author{James East\footnote{Centre for Research in Mathematics, School of Computing, Engineering and Mathematics, Western Sydney University, Locked Bag 1797, Penrith NSW 2751, Australia. {\it Email:} {\tt j.east\,@\,westernsydney.edu.au}} \ and Peter M.~Higgins\footnote{Department of Mathematical Sciences, University of Essex, Colchester CO4 3SQ, UK. {\it Email:} {\tt peteh\,@\,essex.ac.uk}}}

\date{}

\maketitle

\begin{abstract}
We prove new results on inheritance of Green's relations by subsemigroups in the presence of stability of elements.  We provide counterexamples in other cases to show in particular that not all right-stable semigroups are embeddable in left-stable semigroups. This is carried out in the context of a survey of the various closely related notions of stability and minimality of Green's classes that have appeared in the literature over the last sixty years, and which have sometimes been presented in different forms.
\end{abstract}

We take as our starting point two well-known theorems from classical semigroup theory.  Here is the first, which was proved by Green in \cite[Theorem 3]{Green1951}:

\begin{thm}\label{thm:DJ}
In any finite semigroup, we have ${\D}={\J}$.
\end{thm}

The conclusion of Theorem \ref{thm:DJ} is valid if ``finite'' is replaced by ``periodic'' or ``group-bound'' or ``stable''.  For the last, see \cite[Theorem 1]{KW1957} or \cite[Corollary A.2.5]{RSbook}, but note that the definition of stability in \cite{KW1957} and other older papers such as \cite{AHK1965} is slightly different from the modern definition (we will discuss this further at the end of the paper).  We recall the meanings of these terms below, and discuss more general statements still.

Our second motivating theorem concerns the inheritance of Green's relations by subsemigroups.  In all that follows, $U$ will always denote a subsemigroup of a semigroup $S$.  Letting $\G$ stand for any of the five Green's relations, we shall denote~$\G$ on the semigroup~$U$ by~$\Gd$. We use a similar convention for the pre-orders $\leq_{\Gd}$ (for ${\G}\not={\D}$).  Certainly for any $\G$ we have
\[
{\G'}\sub{\G}\restr_U,
\]
where ${\G}\restr_U={\G}\cap(U\times U)$ denotes the restriction of $\G$ to $U$.  While ${\G'}={\G}\restr_U$ need not hold in general, the next result concerns a special case in which it does for the three smaller relations:

\begin{thm}\label{thm:RLH}
If $U$ is a regular subsemigroup of $S$, and if $\G$ is any of $\R$, $\L$ or $\H$, then ${\Gd}={\G}\restr_U$.
\end{thm}

Theorem \ref{thm:RLH} is generally attributed to Hall, cf.~\cite[Result 9(ii)]{Hall1970}, and also \cite[Lemma~1.2.13]{Higgins1992}. The first proof in the literature however is due to Anderson et.~al.~\cite[Proposition 2]{AHK1965} and that is recorded in the book \cite[Proposition 2.4.3]{Lallement1979}.  The following is essentially the argument in \cite[Lemma~2.8]{Sandwiches1}.   We write $\Reg(T)$ for the set of all regular elements of any semigroup $T$.

\begin{lemma}\label{lem:leqR}
Let $x,y\in U$, a subsemigroup of $S$, with $y\in\Reg(U)$. Then $x\leqR y \implies x\leqRd y$. 
\end{lemma}

\pf
We have $x=ya$ for some $a\in S^{1}$. Take $z\in U$ with $y=yzy$.  Then $x=ya=yzya=yzx$.  Since $zx\in U$, this shows that $x\leqRd y$. 
\epf

From Lemma \ref{lem:leqR} (and its dual), it follows that if $x,y\in\Reg(U)$, and if $\G$ is any of $\R$, $\L$ or~$\H$, then $x\G y\implies x\Gd y$.  Theorem \ref{thm:RLH} then follows immediately.

As is well known, the conclusion of Theorem \ref{thm:RLH} does not extend to either the $\D$- or the $\J$-relation, even for a finite regular semigroup $S$ as shown in the first of the following example set.

\begin{egs}
Remembering that ${\D}={\J}$ for finite semigroups (cf.~Theorem \ref{thm:DJ}), let $S=\M^{0}[1;2,2]$ be the $2\times2$ combinatorial (meaning $\H$-trivial) Brandt semigroup, a representation of which is the set of $2\times2$ binary matrices with at most one non-zero entry, under the operation of multiplication: 
\[
S=\{0\}\cup\{a_{ij}:1\leq i,j\leq2\},
\]
where $a_{ij}$ has $1$ as its entry at position $(i,j)$. Observe that $S$ is a five-element inverse semigroup consisting of $a=a_{12}$ and $b=a_{21}$, which are nilpotents, together with three idempotents $0$, $e=a_{11}$ and $f=a_{22}$. The multiplication follows the rule $a_{ij}a_{jk}=a_{ik}$ with all other products being zero. The semigroup $S$ provides the following pair of counterexamples.
\ben

\item The subsemigroup $U=\{0,e,f\}$, being a semilattice, is $\J$-trivial and in particular $(e,f)\not\in{\Dd}$. However, $(e,f)\in{\D}\restr_U$ as $S\setminus\{0\}$ forms a $\D$-class of $S$, with $a$ and $b$ being mutual inverses and $e\R a\L f$ by virtue of the products $ea=a$, $ab=e$, $af=a$ and $ba=f$. Hence the conclusion of Theorem \ref{thm:RLH} is not valid if ${\G}\in\{{\D},{\J}\}$.

\item With the same containing semigroup $S$, take $U=\{0,e,a\}$, which is a subsemigroup of $S$ in which the only non-zero products are $e^{2}=e$ and $ea=a$. Note that $e\in\Reg(U)$ and $e\R a$ in $S$, so certainly $e\leqR a$.  But the equation $au=e$ has no solution for~$u\in U^{1}$, so that $e\not\leqRd a$ (in fact, $a<_{\Rd}e$). So Lemma \ref{lem:leqR} does not hold in general if $y$ (taken here to be $a$) is not regular in $U$, despite $x$ (taken here to be $e$) being regular in~$U$. 

\een
The previous example shows that it is possible to have $x\leqR y$ with $x\in\Reg(U)$ but $x>_{\Rd}y$.  Note that it is never possible to have $x<_{\R}y$ but $x\geq_{\Rd}y$, as the latter implies $x\geq_{\R}y$.  However, it is possible to have $x<_{\R}y$ but $x\not\leqRd y$ with $x\in$ Reg$(U)$ as shown in the next example. 
\ben
\setcounter{enumi}{2}
\item Our semigroup $S$ is $I_{3}$, the symmetric inverse semigroup on the base set $\{1,2,3\}$. Let $f$ be the (partial) mapping defined by $1f=2$ and $2f=3$, and let $g$ be the transposition $(1\,2)$.  Here $3f$ and $3g$ are both undefined. Now $fg=e$ is the idempotent with domain $\{1\}$. Hence in $I_{3}$ we have $e<_{\R}f$.  The inequality is certainly strict since $e<_{\J}f$, as $\rank(e)=1<2=\rank(f)$. Now let $U=\langle f,e\rangle$. Since the elements of $U$ are non-decreasing and $1f>1$ it follows that there is no solution $u\in U^{1}$ to the equation $fu=e$. Hence we have that $e\in\Reg(U)$ and $e<_{\R}f$ but $e\not\leqRd f$. 
\een
\end{egs}

Despite the first of the above examples, we do have the following, which concerns the case in which the regular elements of $S$ form a subsemigroup.  For a proof, see \cite[Lemma 3.8]{Sandwiches1}, the proof of which uses Theorem \ref{thm:RLH}.

\begin{lemma}\label{lem:RegS_D}
If $U=\Reg(S)$ is a subsemigroup of $S$, then ${\Dd}={\D}\restr_U$.
\end{lemma}

From this we may quickly deduce the following:

\begin{cor}\label{cor:RegS_DJ}
If $U=\Reg(S)$ is a subsemigroup of $S$, and if ${\D}={\J}$, then ${\Dd}={\Jd}={\J}\restr_U$.
\end{cor}

\pf
We have ${\Dd}\sub{\Jd}\sub{\J}\restr_U={\D}\restr_U={\Dd}$.
\epf

Now that we have returned somewhat to the theme of Theorem \ref{thm:DJ}, let us work towards a result that extends it, and also gives some kind of analogue of Theorem \ref{thm:RLH} with respect to $\D$- and $\J$-classes.

Recall that a semigroup $S$ is \emph{periodic} if every element has finite order, while $S$ is \emph{group-bound} if for every element $x$ of $S$, some power of $x$ belongs to a subgroup of $S$.  So $S$ is periodic (or group-bound) if for all $x\in S$, there exists $k\geq1$ such that $x^k=x^{2k}$ (or $x^k\H x^{2k}$), respectively, cf.~\cite[Theorems~1.2.2 and~2.2.5]{Howie}.

An element $x$ of a semigroup $S$ is called \emph{right-stable} (or \emph{left-stable}) if for all $y\in S$, we have ${x\J xy\implies x\R xy}$ (or $x\J yx\implies x\L yx$), respectively.  An element is \emph{stable} if it is both left- and right-stable.  A semigroup is \emph{stable} if each of its elements is stable.  Similarly, we may speak of left- or right-stable semigroups.

%Consider now the following sequence of statements, concerning elements $x$ and $y$ of a semigroup $S$:
%\BEN
%\item \label{it:1} If $S$ is finite, then $x\J y \implies x\D y$.
%\item \label{it:2} If $S$ is periodic, then $x\J y \implies x\D y$.
%\item \label{it:3} If $S$ is group-bound, then $x\J y \implies x\D y$.
%\item \label{it:4} If $S$ is stable, then $x\J y \implies x\D y$.
%\item \label{it:5} If at least one of $x,y$ is stable, then $x\J y \implies x\D y$.
%\een
Consider now the following sequence of statements, concerning a semigroup $S$:
\BEN
\item \label{it:1} If $S$ is finite, then ${\D}={\J}$.
\item \label{it:2} If $S$ is periodic, then ${\D}={\J}$.
\item \label{it:3} If $S$ is group-bound, then ${\D}={\J}$.
\item \label{it:4} If $S$ is stable, then ${\D}={\J}$.
\item \label{it:5} If $x,y\in S$, and if at least one of $x,y$ is stable, then $x\J y \implies x\D y$.
\een
Statement \ref{it:1} is Theorem \ref{thm:DJ}, and the others are also well known, cf.~\cite[Proposition~2.1.4]{Howie}, \cite[Theorem 1.2(vi) and Remark 1.7]{HM1979}, \cite[Corollary A.2.5]{RSbook} and \cite[Propositions~2.3.7 and 2.3.9]{Lallement1979}, respectively.  Moreover we have \ref{it:5}$\implies$\ref{it:4}$\implies$\ref{it:3}$\implies$\ref{it:2}$\implies$\ref{it:1}, since
\begin{equation}\label{eq:properties}
\text{finite$\implies$periodic$\implies$group-bound$\implies$stable.}
\end{equation}
Indeed, only the last of these implications is not obvious, so we give a short proof for convenience (for a proof that group-bound semigroups satisfy an alternative formulation of stability, see \cite[Theorem 1.2(vi)]{HM1979}):

\begin{prop}\label{prop:GB_Stable}
Any group-bound semigroup is stable.
\end{prop}

\begin{proof}
Let $S$ be a group-bound semigroup, and let $x,y\in S$.  Clearly it suffices to show that ${x\J xy\implies x\leqR xy}$ and $x\J yx\implies x\leqL yx$.  By duality, it suffices to prove only the first of these implications.  So suppose $x=axyb$ for some $a,b\in S^1$.  Some power of $yb$ belongs to a subgroup of $S$, say $(yb)^k$.  Let $z$ be the inverse of $(yb)^k$ in this subgroup.  Then $x=a^kx(yb)^k = a^kx(yb)^k(yb)^kz = x(yb)^kz \leqR xy$.
\end{proof}

The converse of Proposition \ref{prop:GB_Stable} does not hold.  For example, the semigroup of natural numbers under addition is stable (as is any $\J$-trivial semigroup) but not group-bound.  It is easy to find examples to show that the other implications in \eqref{eq:properties} are also non-reversible in general.

The next result is a generalisation of statement \ref{it:5} above (take $U=S$ in the statement):

\begin{thm}\label{thm:stable}
Let $x$ and $y$ be elements of a semigroup $S$ with $y$ stable, and suppose $x$ and $y$ belong to some subsemigroup $U$ of $S$ for which ${\Ld}={\L}\restr_U$ and ${\Rd}={\R}\restr_U$.  If $x\J y$ and $x\leqJd y$, then $x\Dd y$.
\end{thm}

\pf
Since $x\leqJd y$, we have $x=ayb$ for some $a,b\in U^1$.  It then follows that
\[
x=ayb \leqJ ay \leqJ y \J x,
\]
so that all the above elements are $\J$-related.  In particular, $y\J ay$, and so stability of $y$ gives $y\L ay$, and hence $y\Ld ay$ (as $y,ay\in U$ and ${\Ld}={\L}\restr_U$).  A similar calculation gives $y\Rd yb$.  Since~$\Rd$ is a left congruence, it follows that $ay\Rd ayb=x$.  Thus, $x\Rd ay \Ld y$, and so $x\Dd y$.
\epf

We have already noted that Theorem \ref{thm:DJ} follows from Theorem \ref{thm:stable}, as do each of statements \ref{it:1}--\ref{it:5} above.  We also have the following:

\begin{cor}\label{cor:regstab}
Let $x,y\in U$, a regular subsemigroup of a stable semigroup $S$, with $x\J y$ and $x\leqJd y$.  Then $x\Dd y$. 
\end{cor}

\pf
By Theorem \ref{thm:RLH}, we have ${\Ld}={\L}\restr_U$ and ${\Rd}={\R}\restr_U$.  Since $S$ is stable, Theorem \ref{thm:stable} applies.
\epf

%\begin{rem}
%Corollary \ref{cor:regstab} may seem like a version of Theorem \ref{thm:RLH} for the $\J$ relation, since $x\Dd y$ of course implies $x\Jd y$.  However, it is not as straightforward as that.  Indeed, the $x\leqJd y$ assumption cannot be removed, as shown by Example 
%\end{rem}

We may also infer (by taking $U=S$ in Theorem \ref{thm:stable}) the following fact, proved in \cite[Proposition~2.3.9]{Lallement1979}, which also of course implies that ${\D}={\J}$ in any stable semigroup.  

\begin{cor}
In any semigroup, a $\J$-class containing a stable element is a $\D$-class.
\end{cor}

It was also proved in \cite[Proposition~2.3.7]{Lallement1979} that every element of such a $\J$-class is stable.

\begin{eg}\label{eg:bic}
The stability assumption on $y$ cannot be dropped in Theorem \ref{thm:stable}, due to the fact that any semigroup~$U$ may be embedded in a (regular) bisimple (i.e., $\D$-universal) monoid $S$: see \cite[Theorem~2]{Preston1959} and also \cite[Corollary 1]{Higgins1990}, cf.~\cite[Corollary~1.2.15]{Higgins1992}.  If $U$ here is regular, then certainly ${\Ld}={\L}\restr_U$ and ${\Rd}={\R}\restr_U$, cf.~Theorem \ref{thm:RLH}.  It follows that the two-element semilattice $U=\{0,e\}$ (which is certainly regular) may be embedded in a bisimple semigroup $S$ and then $0\J e$ and $0\leqJd e$, yet $(0,e)\not\in{\Jd}={\Dd}$, contrary to the conclusion of the theorem.  Incidentally, this shows that any bisimple semigroup embedding $U$ is not stable.

 As a concrete example that illustrates the
previous remark we may take~$S$ to be the bicyclic monoid, which
has presentation $S=\langle a,b:ab=1\rangle$. Here $S$ is a bisimple
inverse semigroup, which contains an infinite descending chain of
idempotents: ${1>ba>b^{2}a^{2}>\cdots}$ (here $>$ is the natural partial order defined by $e\leq f \iff e=fef$), and $S$ possesses infinitely many copies of the two-element semilattice.  As noted at the end of the previous paragraph, it follows that the bicyclic monoid $S$ is not stable.  It is also easy to see this directly, as for example $(1,1\cdot b)\in{\J}\setminus{\R}$ and $(1,a\cdot1)\in{\J}\setminus{\L}$.
\end{eg}

Although every semigroup is embeddable in a bisimple semigroup, not every semigroup is embeddable in a stable semigroup, or even in a left- or right-stable semigroup.  The next result (cf.~\cite[Lemma 2.9]{Sandwiches1}) will help clarify this.

\begin{prop}\label{prop:regembed}
If $U$ is a regular subsemigroup of a left- or right-stable semigroup $S$, then $U$ is left- or right-stable, respectively.
\end{prop}

\pf
If $S$ is right-stable (the left-stable case is similar), then for any $x,y\in U$,
%By duality, it is enough to prove just the statement concerning right-stability.  Since $U$ is regular, Theorem \ref{thm:RLH} says that ${\R'}={\R}\restr_U$.  Thus, for any $x,y\in U$,
\[
x\J' xy \implies x\J xy \implies x\R xy \implies x\R' xy,
\]
using the stability of $S$ for the second implication, and Theorem \ref{thm:RLH} for the third.
\epf

%\begin{eg}
%In Example \ref{egs:nonembed}\ref{non3}, we gave an example of a right-stable semigroup that is not embeddable in a left-stable semigroup.  One might wonder if the class of (right- or left-) stable semigroups is closed under taking subsemigroups.  But this is not the case.  For example, in \cite[page 60, Exercise 1]{Howie} we find the following construction.  Write $\RR$ for the set of real numbers,~$\RR^+$ for the set of positive reals, and $S$ for the set of all $2\times2$ matrices of the form $\left[\begin{smallmatrix}a&0\\b&1\end{smallmatrix}\right]$, for $a,b\in\RR^+$.  Then $S$ is a semigroup under ordinary matrix multiplication.  Moreover, $S$ is simple ($\J$-universal) but $\D$-trivial (and hence $\R$-, $\L$- and $\H$-trivial), and has no idempotents.  It follows that for any $x\in S$, $x^2\J x$ holds, yet neither $x^2\R x$ nor $x^2\L x$ hold.  Thus, $S$ is neither left- nor right-stable.  On the other hand, $S$ clearly embeds in the monoid of all $2\times2$ matrices over $\RR$, and it is well known that this monoid is stable, cf.~\cite[Lemma 3.1]{DE2018}.
%\end{eg}

\begin{egs}\label{egs:nonembed}
\ben
\item \label{non1}  The regularity assumption on $U$ cannot be removed in Proposition \ref{prop:regembed}.  Indeed, we can see this by means of a construction from \cite[page 60, Exercise 1]{Howie}.  Let $S$ be the monoid of $2\times2$ real matrices under ordinary matrix multiplication.  It is well known that $S$ is stable, cf.~\cite[Lemma 3.1]{DE2018}.  Let $U$ be the subsemigroup of $S$ consisting of all matrices of the form $\left[\begin{smallmatrix}a&0\\b&1\end{smallmatrix}\right]$ with $a,b>0$.  Then $U$ is simple ($\J$-universal) but $\D$-trivial (and hence $\R$-, $\L$- and $\H$-trivial), and has no idempotents.  It follows that for any $x\in U$, $x^2\J x$ holds, yet neither $x^2\R x$ nor $x^2\L x$ holds.  Thus, $U$ is neither left- nor right-stable.

~ ~ Incidentally, this example shows that the classes of left-stable semigroups and right-stable semigroups are not closed under taking subsemigroups.  On the other hand, examining the properties listed in \eqref{eq:properties}, it is clear that the classes of finite semigroups and periodic semigroups are closed under taking subsemigroups.  The class of group-bound semigroups is not, however.  Indeed, the monoid of natural numbers is a subsemigroup of the group of integers, and the former is not group-bound.

\item \label{non2}  In Example \ref{eg:bic} we observed that the bicyclic semigroup $S=\langle a,b:ab=1\rangle$ is neither left- nor right-stable.  Since $S$ is regular, it follows from Proposition \ref{prop:regembed} that $S$ does not embed in any left- or right-stable semigroup.

%  $U=\langle a,b:ab=1\rangle$ be the bicyclic monoid, which is bisimple, and suppose that~$U$ were embedded in a semigroup $S$.  If $S$ were right-stable, then since $b\J b^{2}$ in $S$ it would follow that $b\R b^{2}$ in $S$ and so there exists $y\in S$ such that $b=b^{2}y$.  But then we would have $1=ab=ab^{2}y=by$, and then $a=aby=y$ so that $ba=by=1$, which is a contradiction.  Dually,~$U$ does not embed in a left-stable semigroup.

~ ~ Corollary 2.2 of \cite{AHK1965} says that the bicyclic semigroup cannot be embedded in a stable semigroup, but again we note that the definition of stability used in \cite{AHK1965} is different from (and stronger than, cf.~Example \ref{eg:OC}) the modern definition we have been using.

\item \label{non3}  It is even possible for a right-stable semigroup not to be embeddable in a left-stable semigroup (and a dual statement also holds).  For example, consider the Baer-Levi semigroup~$U$, consisting of all injective maps $f:\Z\to\Z$ with $\Z\setminus\im(f)$ infinite \cite{BL1932}.  By \cite[Theorem~8.2]{CPbook2}, $U$ is right-simple (i.e.,~$\R$-universal), and hence right-stable.  The same result in \cite{CPbook2} also shows that~$U$ is right-cancellative and without idempotents.  By contrast,~$U$ is not left-cancellative.  Indeed, any left-cancellative and right-simple semigroup is a right group (i.e., the direct product of a group with a right-zero semigroup, cf.~\cite[page 39]{CPbook1}), but $U$ is without idempotents.  The impossibility of embedding $U$ in a left-stable semigroup then follows from the next general fact:
\een
\end{egs}

\begin{prop}\label{prop:embed}
If $U$ is a right-simple semigroup, then the following are equivalent:
\ben
\item \label{embed1} $U$ embeds in a left-stable semigroup,
\item \label{embed2} $U$ is left-stable,
\item \label{embed3} $U$ is left-cancellative,
\item \label{embed4} $U$ is a right group.
\een
\end{prop}

\pf
\ref{embed1}$\implies$\ref{embed3}.  Suppose $U$ embeds in a left-stable semigroup $S$.  
As usual, we write Green's relations on $S$ and $U$ by $\G$ and $\G'$, respectively.  Let $a\in U$.  Then
\[
a\R'a^2 \implies a\R a^2 \implies a\J a^2 \implies a\L a^2,
\]
using left-stability of~$S$ in the last step.  Thus, $a\H a^2$.  It follows that $H_a$ is a group $\H$-class of~$S$ for all $a\in U$, cf.~\cite[Corollary~1.2.6]{Higgins1992}.  

Now suppose $a,b,c\in U$ are such that $ab=ac$.  Let $a^{-1}$ be the inverse of $a$ in the group~$H_a$, and let $e$ be the identity of this group.  Since $b\R' a$, we have $b\R a\R e$, so it follows that $b=eb$.  Similarly $c=ec$.  But then $b=eb=a^{-1}ab=a^{-1}ac=ec=c$.

\ref{embed3}$\implies$\ref{embed4}.  We have already noted that this follows from \cite[page 39]{CPbook1}.

\ref{embed4}$\implies$\ref{embed2}.  Suppose $U=G\times R$, where $G$ is a group and $R$ a right-zero semigroup.  Then $xy\L y$ for all $x,y\in U$.  Indeed, writing $x=(g,s)$ and $y=(h,t)$, we have $(g^{-1},s)\cdot xy=y$.

\ref{embed2}$\implies$\ref{embed1}.  This is clear.
\epf

We next classify the stable (bi)simple semigroups.  Recall that a semigroup is \emph{completely regular} if it is a union of groups: i.e., each $\H$-class is a group.  A semigroup is \emph{completely simple} if it is simple and completely regular.  (This is not the standard definition of completely simple, but see for example \cite[Theorem 3.3.2]{Howie} for a proof of equivalence with the standard definition, which we will not reproduce here.)  Any completely simple semigroup is isomorphic to a Rees matrix semigroup (without zero) over a group.  See \cite[Theorem 3.3.1]{Howie} for a proof of this fact, and also the definition of Rees matrix semigroups.  Since any Rees matrix semigroup over a group is bisimple (i.e., $\D$-universal), so too is any completely simple semigroup.

\newpage

\begin{prop}\label{prop:simple}
If $S$ is a simple (i.e., $\J$-universal) semigroup, then the following are equivalent:
\ben
\item \label{sim1} $S$ is stable,
\item \label{sim2} $S$ is completely regular,
\item \label{sim3} $S$ is completely simple,
\item \label{sim4} $x\R xy\L y$ for all $x,y\in S$.
\een
\end{prop}

\pf
Since $S$ is simple, $x\J xy\J y$ for all $x,y\in S$, so clearly \ref{sim1}$\iff$\ref{sim4}.  We have also noted that the completely simple semigroups are precisely the completely regular simple semigroups, so \ref{sim2}$\iff$\ref{sim3}.

\ref{sim3}$\implies$\ref{sim4}.
Suppose $S$ is completely simple, and let $x,y\in S$ be arbitrary.  Since $S$ is bisimple, $L_x\cap R_y$ is nonempty, and hence an $\H$-class.  It is therefore a group (by complete regularity), and hence contains an idempotent.  It follows from \cite[Proposition 2.3.7]{Howie} that $xy\in R_x\cap L_y$.

\ref{sim4}$\implies$\ref{sim2}.  Take $x=y$ to deduce that $x\H x^2$ for all $x\in S$, so that $H_x$ is a group for all $x$, cf.~\cite[Theorem 2.2.5]{Howie}.
\epf

\begin{rem}
For a simple semigroup $S$, the condition ``$S$ embeds in a stable semigroup'' is not equivalent to conditions \ref{sim1}--\ref{sim4} from Proposition \ref{prop:simple}.  Indeed, the matrix semigroup denoted $U$ in Example \ref{egs:nonembed}\ref{non1} is simple and non-stable, but embeds in a stable semigroup.
\end{rem}

%\begin{eg}
%In Example \ref{egs:nonembed}\ref{non3}, we gave an example of a right-stable semigroup that is not embeddable in a left-stable semigroup.  One might wonder if the class of (right- or left-) stable semigroups is closed under taking subsemigroups.  But this is not the case.  For example, in \cite[page 60, Exercise 1]{Howie} we find the following construction.  Write $\RR$ for the set of real numbers, and $\RR^+$ for the positive reals.  Let $S$ be the monoid of all $2\times2$ matrices over $\RR$ under matrix multiplication, and let $U$ be the subsemigroup of $S$ consisting of all matrices of the form $\left[\begin{smallmatrix}a&0\\b&1\end{smallmatrix}\right]$, for $a,b\in\RR^+$.  Then $U$ is simple ($\J$-universal) but $\D$-trivial (and hence $\R$-, $\L$- and $\H$-trivial).  Since $U$ has no idempotents, we have $x^2\J x$ for all $x\in U$, yet neither $x^2\R x$ nor $x^2\L x$ hold.  It follows that $U$ is neither left- nor right-stable.  On the other hand, it is well known that $S$ (which clearly embeds $U$) is stable, cf.~\cite[Lemma 3.1]{DE2018}.
%\end{eg}

Above we have noted more than once that the older definition of stability \cite{KW1957,AHK1965} is not exactly the same as the modern one we have used.  We conclude with some comments regarding this.  Following Koch and Wallace's older definition \cite{KW1957}, we will say a semigroup $S$ is \emph{KW-stable} if for all $x,y\in S$,
\begin{equation}\label{eq:KW}
Sx\sub Sxy\implies Sx=Sxy \AND xS\sub yxS\implies xS=yxS.
\end{equation}
We will also say that $S$ is \emph{KW$^1$-stable} if for all $x,y\in S$ (or equivalently, all $x,y\in S^1$),
\begin{equation}\label{eq:KW1}
S^1x\sub S^1xy\implies S^1x=S^1xy \AND xS^1\sub yxS^1\implies xS^1=yxS^1.
\end{equation}
Clearly \eqref{eq:KW} and \eqref{eq:KW1} are equivalent if $S^1x=Sx$ and $xS^1=xS$ for all $x\in S$ (which occurs for example if $S$ is regular and/or a monoid).  The latter pair of equalities is equivalent to having $x\in xS\cap Sx$ for all $x\in S$.  Note also that \eqref{eq:KW1} simply says that for all $x,y\in S$,
\begin{equation}\label{eq:KW2}
x\leqL xy \implies x\L xy \AND x\leqR yx\implies x\R yx.
\end{equation}
Part \ref{S1} of the following was proved in \cite[Proposition 2.3.10]{Lallement1979}.  Part \ref{S2} was stated in \cite[footnote~2]{KW1957}.  We provide a simple proof (of \ref{S2}) for completeness.

\begin{prop}
\ben
\item \label{S1} KW$^1$-stability is equivalent to stability.
\item \label{S2} KW-stability implies (KW$^1$-)stability.
\een
\end{prop}

\pf
To prove \ref{S2}, suppose $S$ is KW-stable.  By duality, it suffices to demonstrate the first implication in~\eqref{eq:KW2}, so suppose $x,y\in S$ are such that $x\leqL xy$.  We must show that ${x\geq_{\L} xy}$.  This is clear if $x=xy$, so suppose otherwise.  Since $x\leqL xy$, we have $x\in Sxy$.  Then ${Sx\sub SSxy\sub Sxy}$, so in fact $Sx=Sxy$ by KW-stability.  Since $x\in Sxy=Sx$, it follows that $xy\in Sxy=Sx$, which gives $xy\leqL x$.  
\epf

Although KW-stability implies (KW$^1$-)stability, the converse does not hold in general.  Indeed, the following elegant counterexample was constructed by O'Carroll in \cite[Section 3]{OCarroll1969}:

\begin{eg}\label{eg:OC}
Writing $\N=\{0,1,2,\ldots\}$, define $A = -\N\cup(2\N+1)$ and $B=-\N$.  Also define $f,g\in\T_\Z$ (the full transformation semigroup on $\Z$) by 
\[
f: \begin{cases}
a\mapsto 0 &\text{for $a\in A$}\\
2x\mapsto -x &\text{for $x\geq1$}
\end{cases}
\AND
g: \begin{cases}
b\mapsto b+1 &\text{for $b\in B$}\\
2x\mapsto 2x+1 &\text{for $x\geq1$}\\
2x-1\mapsto 2x+1 &\text{for $x\geq1$.}
\end{cases}
\]
Since $g$ is injective on $A=\im(g)$, it follows that $g$ is a group element of~$\T_\Z$.  Let $G$ be the infinite cyclic group generated by $g$.  Also write $K$ for the kernel of $\T_\Z$, which consists of all constant mappings $\Z\to\Z$.  Then by \cite[Theorem 3.5]{OCarroll1969}, $S=G\cup f G\cup K$ is a (KW$^1$-)stable semigroup that is not KW-stable.
\end{eg}

Consider again the (modern) definition of stability, which says that for all $x,y\in S$,
\begin{equation}\label{eq:stab}
x\J xy \implies x\R xy \AND x\J yx \implies x\L yx.
\end{equation}
It follows quickly from the first of these implications that if $x$ and $z$ are two elements of some common $\J$-class of $S$, then $R_z\leqR R_x\implies R_z=R_x$ (if $R_z\leqR R_x$, then $z=xy$ for some $y\in S^1$, so $x\J xy$ and we apply \eqref{eq:stab}).  That is, if $S$ is right-stable, then for any $\J$-class $J$ of $S$, all $\R$-classes contained in $J$ are minimal in the $\leqR$-ordering on $\R$-classes contained in $J$.  This clearly implies the following condition:
\begin{itemize}
\item[$M_R^*$:] ~ For each $\J$-class $J$ of $S$, the set of all $\R$-classes contained in $J$ has a minimal element.
\end{itemize}
The reverse implication ($M_R^*$$\implies$right-stability) was proved in \cite[Lemma 2.2]{Munn1957}.  Of course left-stability is equivalent to condition $M_L^*$, defined dually in terms of $\L$-classes.  See also \cite{HM1979}, \cite[Section 6.6]{CPbook2} and \cite[Section 1.2]{Higgins1992} for further discussions of these and other minimality conditions.

Condition $M_R^*$ is weaker than the condition known as $M_R$.  The semigroup $S$ satisfies $M_R$ if every nonempty set of $\R$-classes of $S$ contains an element minimal in the $\leqR$ order, or equivalently that there are no infinite descending chains of $\R$-classes.  (Condition $M_L$ is defined dually with respect to $\L$-classes.)  Green's main extension of Theorem \ref{thm:DJ} to classes containing infinite semigroups is \cite[Theorem 8]{Green1951}, which states that ${\D}={\J}$ if $S$ satisfies $M_R$ and $M_L$ (or, in Green's notation, that $\mathfrak d=\mathfrak f$ if $S$ satisfies $\mathscr M_{\mathfrak r}$ and $\mathscr M_{\mathfrak l}$).  Since we have noted that the conjunction of $M_R^*$ and $M_L^*$ is equivalent to stability, we have the following well-known result (cf.~\cite[Theorem 6.45]{CPbook2}):

\begin{thm}
If $S$ satisfies $M_R^*$ and $M_L^*$, then ${\D}={\J}$.  
\end{thm}

On the other hand, there are no one-sided versions of the results discussed in the previous paragraph.  Indeed, if $S$ is a so-called \emph{Croisot-Teissier semigroup}, as defined in \cite[Section~8.2]{CPbook2}, and first studied in \cite{Croisot1954,Teissier1953}, then $S$ satisfies $M_R$ (and hence $M_R^*$), but not $M_L^*$ (and hence not~$M_L$), yet ${\D}\not={\J}$ in general.  This all follows from \cite[Theorem~8.11]{CPbook2} and \cite[Theorems~6,~11 and~12]{Levi1986}, cf.~\cite[Section~1]{HM1979}.

\subsection*{Acknowledgements}

We thank the referee for some valuable suggestions, especially for suggesting Proposition \ref{prop:simple} and strengthening our previous version of Proposition \ref{prop:embed}.

\newpage

\footnotesize
\def\bibspacing{-1.1pt}
\bibliography{biblio}
\bibliographystyle{abbrv}

\end{document}